\documentclass[12pt]{article}
 \usepackage{a4wide}
\def\myfont{\fontfamily{cmss}\selectfont}

\usepackage{fancyhdr, amsmath, amssymb, amsfonts, url, mathrsfs, graphicx, epsfig,  amsthm}
\usepackage{tikz}
\usepackage{amscd}

\numberwithin{equation}{section}

\newtheorem{thm}{Theorem}[section]

\newtheorem{lemma}[thm]{Lemma}
\newtheorem{proposition}[thm]{Proposition}
\newtheorem{definition}[thm]{Definition}
\newtheorem{example}[thm]{Example}
\newtheorem{remark}[thm]{Remark}

\usepackage{pgfplots}

\AtEndDocument{\bigskip{\footnotesize
  \noindent
  \textsc{D. Parmenter, School of Mathematics, University of Bristol, Bristol, BS8 1UG, UK, and the Heilbronn Institute for Mathematical Research, Bristol, UK.}
  
  \noindent
  \textit{E-mail address}: \texttt{d.parmenter@bristol.ac.uk}
  \par
  \addvspace{\medskipamount}
  \noindent
  \textsc{M. Pollicott, Department of Mathematics, Warwick University, Coventry, CV4 7AL, UK}
  
  \noindent
  \textit{E-mail address}: \texttt{masdbl@warwick.ac.uk}
}}

\begin{document}
 \title{Constructing equilibrium states for some partially hyperbolic attractors via  densities}
\author{David Parmenter and Mark Pollicott
\thanks{
The second author is partly
supported by ERC-Advanced Grant 833802-Resonances and EPSRC grant
EP/T001674/1}
}
\maketitle
 \myfont

\abstract{We shall describe a new construction of equilibrium states for a class of partially hyperbolic systems. This generalises our construction for Gibbs measures in the uniformly hyperbolic setting. This more general setting introduces new issues that we need to address carefully, in particular requiring additional assumptions on the transformation. We treat two cases: either the centre-stable manifold satisfies a bounded expansion condition considered in \cite{CPZ-2};  or the centre-unstable manifold satisfies a subexponential contraction condition  which appears new in the context of equilibrium state constructions. The problem of constructing equilibrium states was previously raised by Pesin-Sinai \cite{ps}  and Dolgopyat \cite{Dogo} for the particular case of u-Gibbs measures, and by Climenhaga, Pesin and Zelerowicz  \cite{CPZ-2} for other equilibrium states.}

\section{Introduction}

In this article we will describe a new construction of equilibrium states for certain partially hyperbolic diffeomorphisms. In particular, we will extend the ideas in \cite{PaPo} to the setting of appropriate partially hyperbolic diffeomorphisms. 

Given a continuous map on a compact metric space $f: X \to X$ and a continuous function $G: \textcolor{black}{X} \to \mathbb R$ we say that a $f$-invariant probability measure  $\mu_G$ is an equilibrium state for $G$ if 
$$h(\mu_G) + \int G d\mu_G = \sup\left\{h(\mu) + \int G d\mu \hbox{ : } \mu \in \mathcal M_f(X) \right\} \eqno(1.1)$$
where $\mathcal M_f(X)$ denotes the space of $f$-invariant probability measures
i.e. $\mu_G$ is a measure which maximizes the sum of the entropy $h(\mu)$ and the integral of $G$,  over all $f$-invariant measures $\mu$ \cite{ruelle}, \cite{walters}.
The value attained by  the supremum in (1.1)  is called the pressure and denoted $P(G)$.

In the uniformly hyperbolic setting if $G$ is (H\"older) continuous then  there exists a (unique) equilibrium state by virtue of expansiveness and there are a variety of different constructions of such measures.  A particularly well known  example is the SRB-measure
which induces absolutely continuous measure with respect to the induced volume on unstable manifolds.
In the partially hyperbolic setting the constructions are more delicate.
Fundamental work due to Pesin and Sinai \cite{ps} extends the notion of an SRB measure for uniformly hyperbolic attractors to u-Gibbs measures in the partially hyperbolic setting. 
Intuitively u-Gibbs measures are defined to be absolutely continuous with respect to the induced volume on unstable manifolds 
(cf. \cite{Dogo}). In \cite{ps} the authors start with the normalised induced volume on a piece of local unstable manifold then taking averaged pullbacks, the limiting measures are $f$-invariant u-Gibbs measures.



In this article we consider a construction of equilibrium states for partially hyperbolic diffeomorphisms $f: X \to X$ by using  sequences of  measures supported on a piece of unstable manifold $W_\delta^u(x)$ that are absolutely continuous with respect to the induced volume. In \cite{PaPo} we replaced the induced measure on unstable leaves for hyperbolic attractors, denoted $\lambda$, by the integral with respect to $\lambda$ over $W_\delta^u(x)$ of specially chosen density functions. In this article we will apply similar methods as those in \cite{PaPo} to the setting of partially hyperbolic diffeomorphisms (although we require additional conditions that will be described below). Let $\Phi : X \rightarrow \mathbb{R}$ be the geometric potential defined by $\Phi(x) = - \log | $det$(Df|E_x^u)|,$ where $E_x^u$ is the unstable bundle for the diffeomorphism. This is usually called the unstable expansion coefficient. 

\begin{thm}\label{thm:gibbsdiffeo_intro}
Let $f: X \to X$ be a $C^{1+\alpha}$ topologically mixing partially hyperbolic attracting diffeomorphism satisfying Lyapunov stability (Definition \ref{LS}) and  let  $G: X \to \mathbb R$ be a continuous function.
Given $x\in X$ and $\delta > 0$ 
consider 
the  sequence of probability measures $(\lambda_n)_{n=1}^\infty$ 
supported on $W^u_\delta(x)$ 
and absolutely continuous with respect to the induced volume $\lambda = \lambda_{W^u_\delta(x)}$
with densities
$$
\frac{d\lambda_n}{d\lambda}(y) :=  
\frac{
 \exp\left( \sum_{i=0}^{n-1} (G - \Phi)(f^iy)) \right) 
 }{
 \int_{W_\delta^u(x)}
 \exp\left( \sum_{i=0}^{n-1} (G-\Phi)(f^iz) \right) 
 d\lambda(z)
 }
\quad  \hbox{ for } y \in  W^u_\delta(x).\eqno(1.2)
$$
Then the weak* limit points of the averages  
$$\mu_n :=\frac{1}{n}\sum_{k=0}^{n-1} f_*^k \lambda_n, \quad n \geq 1,
\eqno(1.3)
\label{mu_n}
$$
(where $f_* \lambda_n(A) = \lambda_n(f^{-1}A)$ for Borel sets $A \subset X$) are equilibrium states for $G$.
\end{thm}

From the change of variables for
 $f^n: W^u_\delta(x) \to f^n(W^u_\delta(x))$, by the chain rule the Jacobian is
$$\left|\det (Df^n|E_y^u)\right|
= \prod_{i=1}^n \left|\det  (Df|E_{f^iy}^u)\right| = 
 \exp\left( -\sum_{i=0}^{n-1}  \Phi(f^iy) \right).
 $$
Thus we can reformulate (1.2) as 
 $$
 \lambda_n(A) = 
 \frac{\int_{f^{\textcolor{black}{n}}(A)} \exp \left( \sum_{i=1}^{n} G(f^{-i}y)\right) d\lambda_{f^nW_\delta^u}(y)}
 {\int_{f^nW_\delta^u(x)} \exp \left( \sum_{i=1}^{n} G(f^{-i}z)\right) d\lambda_{f^nW_\delta^u}(z)}
 \hbox{  for Borel  $A \subset W_\delta^u(x)$}
\eqno(1.4)
 $$ 
 which is often  more convenient  in the proofs.

The proof of Theorem \ref{thm:gibbsdiffeo_intro} relies on the following characterisation of the pressure which is of independent interest.
\begin{proposition}\label{pressure_intro}
Let $f: X \to X$ be a topologically mixing partially hyperbolic attracting diffeomorphism satisfying Lyapunov stability. 
For any  continuous function $G: X \to \mathbb R$
we have 
$$
P(G)
= \limsup_{n \to +\infty} \frac{1}{n}\log  \int_{W_\delta^{u}(x)}
 \exp\left( \sum_{k=0}^{n-1} (G-\Phi)(f^ky) \right) 
 d\lambda_{W_\delta^u(x)}(y). $$
\end{proposition}


Another interesting construction of equilibrium states for partially hyperbolic diffeomorphisms can be found in \cite{CPZ-2}. There the authors, Climenhaga, Pesin and Zelerowicz, construct unique equilibrium states by a Carath\'eodory type construction using fixed reference measures on pieces of local unstable manifold. They require the following (heuristic) additional conditions on a $C^2$ diffeomorphism $f : M \rightarrow M$:
\begin{itemize}
\item the map $f$ is topologically transitive and partially hyperbolic with integrable centre-stable bundle along which the expansion under $f^n$ is uniformly bounded independently of $n$ (which we call Lyapunov stability); and
\item the potential $G$ satisfies a `Bowen property' that uniformly bounds the difference between Birkhoff sums along nearby trajectories, independently of the trajectory length.
\end{itemize}

Additionally, in the setting of partially hyperbolic diffeomorphisms with centre isometries, Carrasco and Rodriguez-Hertz \cite{C-RH} provide a geometric construction of equilibrium states for H\"older potentials that are either the geometric potential or are constant along centre leaves. 

We finish this article with an extension of our construction to 
systems with a stable, centre-unstable splitting. 
There we  require a hypothesis  which is  less restrictive than Lyapunov stability. 
We require subexpontential contraction in centre-unstable manifolds (see Definition \ref{secicum}). 


\section{Definitions}

In this section we introduce the notion of partially hyperbolic dynamics and provide useful background results and classical examples. We then briefly recall the definition of Kolmogorov-Sinai entropy and state a useful result due to Misiurewicz.

\subsection{Partial Hyperbolicity}

We begin with the definition of partial hyperbolicity.

\begin{definition}[Partially hyperbolic set] \label{parthyp}
Let $M$ be a compact Riemannian manifold and $f : M \rightarrow M$ be a diffeomorphism. A closed $f$-invariant subset $X \subset M$ is said to be partially hyperbolic (in the broad sense) if:
\begin{itemize}
\item there exists a continuous splitting of the tangent bundle into subbundles $E^u$ and $E^{cs}$ such that $T_{X}M = E^u \oplus E^{cs}$; and
\item there is a Riemannian metric $|| \cdot ||$ on $M$ and constants $C_1 > 1$ and $0 < C_2 < C_1$ such that for every $x \in X$
\begin{align*}
|| Df_x v|| & \geq C_1 ||v|| \hbox{ \hspace{5mm} for $v \in E^u(x)$}, \\
|| Df_x v|| & \leq C_2 ||v|| \hbox{ \hspace{5mm} for $v \in E^{cs}(x)$}.
\end{align*}
\end{itemize}
\end{definition}

In the above definition we are using the property that the Riemannian metric can be chosen to be adapted (see \cite{HPsur}, \textsection 2.1.2). We shall also always assume that the map is topologically mixing. 

 If $C_2 < 1$ then $f$ is uniformly hyperbolic. Definition \ref{parthyp} also covers the more traditional and narrower definition of partial hyperbolicity involving a splitting into an expanding, centre and contracting direction. 


There is a continuous cone family $K^{cs}$ defined on $X$ which is $Df^{-1}$-invariant (\cite{HPsur}, Proposition 2.9). A curve $\gamma$ is a $cs$-curve if all its tangent vectors lie in $K^{cs}$.  We can now give the formal definition of Lyapunov stability.

\begin{definition} \label{LS}
(Lyapunov Stability) A partially hyperbolic diffeomorphism $f : X \rightarrow X$ has 
\emph{Lyapunov stability} if for every $\epsilon>0$ there exists an $\epsilon_0>0$ such that if $\gamma$ is a curve in $X$ with length at most $\epsilon_0$ and $n \geq 0$ is such that $f^n \gamma$ is a $cs$-curve then the length of $f^n \gamma$ is at most $\epsilon$.
\end{definition}

This condition is labelled (C1) in \cite{CPZ-2}. This property is important because unlike in the uniformly hyperbolic setting, it is possible that there exist multiple equilibrium states if this condition is not satisfied.

\begin{remark}
We note that in this article we will only be concerned with the existence of equilibrium states but it is worth observing that in \cite{CPZ-2} an example is provided which does not satisfy Lyapunov stability and has at least 2 measures of maximal entropy (for $G=0$).
\end{remark}

\subsection{Local manifolds}

We begin with the definition of a local unstable manifold.

\begin{definition}[Local unstable manifold] \label{LUM}
Choose $\rho > 0$ such that $C_2 < 1/\rho < C_1$. For $x \in X$ and sufficiently small $\delta > 0$ we define the local unstable manifold through $x$ as follows,
$$
W_\delta^{u}(x) = \{y \in M \hbox{ : } d(f^{-n}x, f^{-n}y) \leq \delta, \forall n \geq 0 \hbox{ and } d(f^{-n}x, f^{-n}y) / \rho^n \rightarrow 0 \hbox{ as }n \rightarrow \infty \}.
$$
\end{definition}

Intuitively these are the points whose backwards orbits stay $\delta$-close to that of $x$ and whose separation depends on $\rho$. The following is a useful version of the unstable manifold theorem.

\begin{thm} (\cite{shub}, Theorem IV.1.) \label{stab}
For $f \in C^r(M)$, $r > 1$, and $\delta > 0$ sufficiently small, there exists a $\lambda \in (C_1^{-1}, 1)$ and $C > 0$ such that for any $x \in X$ there is a $C^{r}$ local unstable manifold $W_{loc}^u(x) \subset M$ such that,
\begin{itemize}
\item $T_yW_\delta^u(x) = E^u(y)$ for every $y \in W_\delta^u(x) \cap X$;
\item $W_\delta^{u}(x)$ is a $C^{r}$ embedded  disk of dimension $\dim (E^u)$; and
\item for every $n \geq 0$ and $y \in W_\delta^u(x)$, $d(f^{-n}x, f^{-n}y) \leq C \lambda^{n} d(x,y)$.
\end{itemize}
\end{thm}



In the uniformly hyperbolic setting we are able to define a similar local stable manifold focusing on the forward orbits of points but it is not always possible in the partially hyperbolic setting. The following theorem states that with the assumption of $f$ satisfying Definition \ref{LS}, there is a local centre-stable manifold that plays the role of the stable manifold from the uniform hyperbolic setting.

\begin{thm} \label{LPS}
(Local product structure, \cite{CPZ-2}, Theorem 2.4) Let $f : M \rightarrow M$ be a diffeomorphism and $X \subset M$ be a compact $f$-invariant subset admitting a splitting $E^{cs} \oplus E^u$ satisfying Lyapunov stability. 
For every $x \in X$ there is a local manifold $W_{loc}^{cs}(x) \subset M$ satisfying:
\begin{itemize}
\item $T_yW_{loc}^{cs}(x) = E^{cs}(y)$ for every $y \in W_{loc}^{cs}(x) \cap X$; and 
\item $W_{loc}^{cs}(x)$ is a $C^{1}$ embedded  disk of dimension $\dim (E^{cs})$.
\end{itemize}
Moreover, there exists an $\epsilon > 0$ such that if $x, y \in X$ and $d(x,y) < \epsilon$ then $W_\delta^u(x) \cap W_{loc}^{cs}(y)$ is exactly one point which we denote $[x,y]$.
\end{thm}



\subsection{Examples}

This subsection contains some classic examples of partially hyperbolic diffeomorphisms. 

Linear hyperbolic toral automorphisms are the most basic examples of Anosov diffeomorphisms and have no eigenvalues on the unit circle. We begin with an extension of these to partially hyperbolic diffeomorphisms by relaxing this condition.

\begin{example}[Quasihyperbolic linear toral automorphism]
Let $M = \mathbb{T}^n$, $T_A : M \rightarrow M$ be a linear map corresponding to an $n\times n$ integer valued matrix $A$  with $\det A = 1$ and such that $A$ has an eigenvalue which is not a root of unity.
The unstable direction corresponds to the eigenspace for the eigenvalues of absolute values greater than 1. The centre-stable direction corresponds to the eigenspace for eigenvalues sitting on or inside the unit circle. In fact the centre-stable direction could be chosen to be expanding in some directions if the eigenvalues are greater than 1 and strictly less then the eigenvalues in the specified unstable direction. However, in this case the maximal expansion of the centre-stable bundle would have exponential growth rate and so would not satisfy the Lyapunov stability condition. 
\end{example}



\begin{example}[Compact group extensions]
Consider an Anosov diffeomorphism $f : M \rightarrow M$ and $G$ a compact Lie group. Now consider a smooth map $\varphi : M \rightarrow G$. We can define the map $F : M \times G \rightarrow M \times G$ by,
\begin{equation*}
F(x,y) = (f(x), \varphi(x) y) \hbox{ \hspace{5mm} for $x \in M, y \in G$}.
\end{equation*}
Left translations are isometries of $G$ in the bi-invariant metric and so this map is partially hyperbolic with $G$ corresponding to the neutral direction.
\end{example}

\begin{example}[Time 1-Anosov flow]
Let $\phi_t : M \rightarrow M$ be an Anosov flow (\textsection 17.4, \cite{kh}) on a Riemannian manifold, $M$. By definition the tangent bundle exhibits a continuous splitting, $TM = E^u \oplus E^0 \oplus E^s$ with expanding direction $E^u$, stable direction $E^s$ and $1$-dimensional flow direction $E^0$. The time one map, $\phi_1$, is a partially hyperbolic diffeomorphism. 
\end{example}

A classic result due to Anosov \cite{anosov} states that the geodesic flow on a negatively curved manifold is an Anosov flow. The following example is a compact group extension of the geodesic flow.


\begin{example}[Time 1-Frame flow]
Let $V$ be a closed oriented $n$-dimensional manifold of negative sectional curvature and $M = SV$, the unit tangent bundle of $V$. Let $N$ be the space of positively oriented orthonormal $n$-frames in $TV$, the tangent bundle of $V$. This produces a fiber bundle $\pi : N \rightarrow M$ where $\pi$ is the natural projection onto the first vector. 
We can identify each fiber $N_x$ with $SO(n-1)$. Define the frame flow $\Phi_t$ by mapping the first vector forward under the geodesic flow and the other vectors by parallel translation along the geodesic defined by the first vector. Notice that for any $t$, $\pi \circ \Phi_t = g_t \circ \pi$, where $g_t$ is the geodesic flow. The time one map arising from the frame flow has centre bundle with dimension $1+\dim SO(n-1)$ and unstable/stable directions coming from the geodesic flow.
\end{example}

We see that all of the above examples have neutral centres and as such all satisfy the Lyapunov stability condition. Therefore, having this condition does not appear to be too restrictive.

\subsection{Entropy}

We now describe  some results on the entropy of invariant probability measures.  For the duration of this subsection we will only require 
the weaker assumption that $f: X \to X$ is a homeomorphism.

We begin with some standard definitions, \cite{walters}.

\begin{definition}
Given a finite measurable partition $\mathcal P = \{P_1, \cdots, P_k\}$ and a probablity measure $\nu$ we can associate the
 entropy of the partition defined by 
 $$
 H_\nu(\mathcal P)
 =  - \sum_{i=1}^k  \nu(P_i)  \log \nu(P_i). 
 $$
 \end{definition}
 
 \noindent
 Given $n \geq 1$ we let $\bigvee_{i=0}^{n-1} f^{-i} \mathcal P = \{P_{i_0} \cap f^{-1} P_{i_1} \cap \cdots  \cap f^{-(n-1)} P_{i_{n-1}}
 \hbox{ : } 1 \leq  i_0, \cdots, i_{n-1} \leq k
 \}$ be the refinement of the partitions $\mathcal P, f^{-1}\mathcal P, \cdots, f^{-(n-1)}\mathcal P$.

 \begin{definition}
We can  define the entropy associated to the partition $\mathcal P$ by 
 $$
 h_\nu(\mathcal P) = \lim_{n \to +\infty} \frac{1}{n} H_\nu\left( \bigvee_{i=0}^{n-1} f^{-i} \mathcal P\right).
 $$
Finally, the entropy with respect to the measure is defined by
$$
h(\nu) = \sup_{\mathcal{P}}\{h_{\nu}(\mathcal{P}) \hbox{ : $\mathcal{P}$ is a countable partition with $H_{\nu}(\mathcal{P}) < \infty$}\}.
$$
In the case that $\mathcal P$ is a generating partition we have that $h(\nu) = h_\nu(\mathcal P)$ is the entropy of the measure $\nu$.
 \end{definition}

We require the following application of a lemma due to Misiurewicz \cite{misiurewicz}, relating the entropy with respect to $\lambda_n$ and the averaging of pushforwards $\mu_n$ (defined in (1.2) and (1.3) respectively).

  \begin{lemma} \label{Misent} For any $n \geq 2$ and $0 < q < n$,
  $$
  q H_{\lambda_n} \left(
  \bigvee_{h=0}^{n-1} f^{-h} \mathcal P 
 \right)
 \leq 
n   H_{\mu_n}\left(
  \bigvee_{i=0}^{q-1} f^{-i} \mathcal P 
 \right)
 +
 2q^2 \log\hbox{\rm Card} (\mathcal P).
  $$
  \end{lemma}

We shall not provide a proof here but one can be found in \cite{walters}, p.220.

\section{Pressure and growth}

The pressure $P(G)$ has various different interpretations in terms of the growth of appropriate quantities. We begin with the standard definition of the pressure $P(G)$ written in terms of the growth rates of sets of spanning sets and separated sets due to Bowen.

Recall that given $\epsilon > 0$ and $n \geq 1$,  an $(n,\epsilon)$-spanning set $S \subset X$ is such that $ \cup_{x\in S} B(x,n,\epsilon)$ covers $X$,  where $B(x,n,\epsilon) := \cap_{k=0}^{n-1}f^{-k} B(f^kx, \epsilon)$  is called a   Bowen ball. 
On the other hand  a\textcolor{black}{n} $(n,\epsilon)$-separated  set $\Sigma \subset X$ 
is such that \textcolor{black}{$d_n(x,y) > \epsilon$} for $x,y \in \Sigma$
 (and, in particular, $B(x,n,\epsilon/2)$, $x \in \Sigma$, are disjoint
in $X$).

\begin{lemma}  \label{defpressure}
Given 
$n \geq 1$ and $\epsilon > 0$,  let 
$$
Z_{0}(n, \epsilon) = \inf\left\{
\sum_{y \in S} \exp(G^n(y))
\hbox{ : } S \hbox{ is an $(n, \epsilon)$-spanning set }
\right\}
$$
and 
$$
Z_{1}(n, \epsilon) = \sup\left\{
\sum_{y \in \Sigma } \exp(G^n(y))
\hbox{ : } \Sigma  \hbox{ is an $(n, \epsilon)$-separated set }
\right\}.
$$
where $G^n(x) = \sum_{j=0}^{n-1}G(f^jx)$.
Then the following limits exist and are equal
$$
P(G):=
\lim_{\epsilon \to 0}
\limsup_{n \to +\infty} \frac{1}{n}\log Z_{0}(n, \epsilon)
 = \lim_{\epsilon \to 0}
\limsup_{n \to +\infty} \frac{1}{n}\log Z_{1}(n, \epsilon).
$$
\end{lemma}
(See \cite{walters}, chapter 9)

\begin{definition}
    For a continuous potential $G : X \rightarrow \mathbb{R}$ the value $P(G)$ 
    is called the {\it topological pressure} of $G$. 
\end{definition}

We will need the following characterisation of the pressure in terms of growth rates of appropriate densities on a piece of unstable manifold.

\begin{proposition} \label{pressure}
Let $f: X \to X$ be a mixing partially hyperbolic attracting diffeomorphism satisfying Lyapunov stability. 
For any  continuous function $G: X \to \mathbb R$
we have 
$$
P(G)
= \limsup_{n \to +\infty} \frac{1}{n}\log  \int_{W_\delta^{u}(x)}
 \exp\left( \sum_{k=0}^{n-1} (G-\Phi)(f^ky) \right) 
 d\lambda_{W_\delta^u(x)}(y). \eqno(3.1)$$
\end{proposition}

We can  use the change of variables formula  to rewrite this in the equivalent form
$$
P(\textcolor{black}{G})
= \limsup_{n \to +\infty} \frac{1}{n}\log  \int_{f^nW_\delta^{u}(x)}
\exp\left( \sum_{k=1}^{n} G(f^{-k}y) \right)
  d\lambda_{f^nW_\delta^u(x)}(y). . \eqno(3.2)
 $$

\begin{example}
If we consider the potential $G = 0$ then Proposition \ref{pressure} gives
\begin{equation*}
P(0) =  \limsup_{n \rightarrow + \infty} \frac{\log \lambda(f^n W_{\delta}^u(x))}{n}.
\end{equation*}
Therefore, the topological entropy $h_{top}(f) = P(0)$ is the exponential growth rate of the volume of a small piece of unstable manifold.
\end{example}

\begin{example}
Another important potential is the geometric potential $G = \Phi = -  \log \left|\det ( Df|E_x^u)\right|$. Then Proposition \ref{pressure} gives, 
\begin{align*}
P(\Phi) & = \limsup_{n \to +\infty} \frac{1}{n}\log  \int_{W_\delta^{u}(x)}
 \exp\left( \sum_{k=0}^{n-1} (\Phi -\Phi)(f^ky) \right) 
 d\lambda_{W_\delta^u(x)}(y) \\
 & = \limsup_{n \to +\infty} \frac{1}{n}\log  \lambda(W_\delta^{u}(x))
	 = 0.
\end{align*}
This is a well known result in the uniformly hyperbolic setting (\cite{Bowen}, Theorem 4.11)
and we provide a new elementary proof of this fact for partially hyperbolic diffeomorphisms.
\end{example}

\begin{example}
More generally we can consider the geometric q-potential defined by, 
\begin{equation*}
\Phi_q = - q \log \left|\det ( Df|E_x^u)\right| \text{ \hspace{10mm} for $q \in \mathbb{R}$}.
\end{equation*}
Then the pressure has the following  characterisation,
\begin{equation*}
P(\Phi_q) = \limsup_{n \rightarrow \infty} \frac{1}{n} \log  \int_{W_\delta^u(x)} \left|\det (Df^n|E_z^u)\right|^{1-q}
  d\lambda_{W_\delta^u(x)}(z).
\end{equation*}
\end{example}

The proof of Proposition \ref{pressure} is analogous to the proof of Proposition 4.6 in \cite{PaPo}. The proof in \cite{PaPo} relies on the local product lemma and the requirement that points on local unstable manifolds contract in the past. We describe the adjustments to the partially hyperbolic setting using the local product structure in Lemma \ref{LPS} and the assumption that we have some control over the expansion of points under $f^n$ in the centre-stable direction.

 \begin{proof}[Proof of Proposition \ref{pressure}]
 We begin with the following standard result, which can be compared with Lemma 4.7 in \cite{PaPo}.

\begin{lemma}\label{cover0}
{F}or any $\epsilon_0 > 0$ there exists an $m > 0$ such that $f^mW^u_\delta(x)$ 
is $\epsilon_0$-dense in $X$.  In particular, we can assume that $X = \cup_{y \in f^mW^u_\delta(x)} W_{loc}^{cs}(y).$
\end{lemma}
\begin{proof} [Proof of Lemma \ref{cover0}]
For $\epsilon_1 > 0$ small enough, consider an $\epsilon_1$-fattening, $W_{\delta, \epsilon_1}^u$, of $W_{\delta}^u(x)$. By this we mean $W_{\delta, \epsilon_1}^u = \cup \{ B(y, \epsilon_1) \text{ : } y \in W_\delta^u(x)\}$. For any $z \in W_{\delta, \epsilon_1}^u$, using Lemma \ref{LPS} we have $W_\delta^u(x) \cap W_{loc}^{cs}(z)$ is a single point which we call $z_1$. Providing  $\epsilon_1 > 0$ is small enough then by Lyapunov stability we deduce that $d(f^m z_1, f^m z) < \epsilon_0/2$ for any $m \in \mathbb{N}$. Thus we have shown, for any $z \in W_{\delta, \epsilon_1}^u$, there exists a $z_1 \in W_\delta^u(x)$ such that $d(f^m z_1, f^m z) < \epsilon_0/2$ for any $m \in \mathbb{N}$.

For any $\epsilon_0 > 0$, by using compactness we can choose a finite set $C$ 
such that $\{B(c,\epsilon_0/4) \text{ : } c \in C\}$ forms a cover of $X$. Since $W_{\delta, \epsilon_1}^u$ is open, by topological mixing 
providing $m \in \mathbb N$ is sufficiently large, for 
 any $c \in C$ $f^m W_{\delta, \epsilon_1}^u \cap B(c, {\epsilon_0}/{4}) \neq \emptyset$.
 Let $y_c \in f^m W_{\delta, \epsilon_1}^u \cap B(c, {\epsilon_0}/{4})$ then for any $y \in X$ there is a $c \in C$ such that $y \in B(c,\epsilon_0/4)$ and in particular, $d(y, y_c) < \epsilon_0/2$. We have $f^{-m}y_c \in W_{\delta, \epsilon_1}^u$ and therefore by the preceeding paragraph there is a $z \in W_\delta^u(x)$ such that $d(y_c , f^m z) < \epsilon_0/2$.

By the triangle inequality, for any $y \in X$, we have that  there is a $z \in W_\delta^u(x)$ such that $d(y, f^m z) \leq d(y, y_c) + d(y_c,f^m z) < \epsilon_0$. Taking $\epsilon_0$ small enough (and making $\epsilon_1$ smaller if necessary), we can apply the local product structure (Lemma \ref{LPS}) again to conclude.
\end{proof}


To get a lower   bound on the growth rate in Proposition \ref{pressure} we proceed as follows.
 Given $\epsilon > 0$ and 
  $n \geq 1$  we want to 
construct an $(n, 2\epsilon)$-spanning set.  
We begin by choosing a covering of $ f^{n+m} W_\delta^u(x)$
by 
 $\epsilon$-balls 
$$
B_{d_u}(x_i,  \epsilon) \hbox{ : } i=1, \cdots, N:=N(n+m, \epsilon)$$
contained within the unstable manifold with respect to  the induced metric denoted by $d_u$
and let  $A_{\epsilon} : = f^{n+m} W_\delta^u(x) \setminus \bigcup_{y \in \partial f^{n+m} W_\delta^u(x)} B_{d_u}(y, \epsilon/2)$, where $\partial f^{n+m} W_\delta^u(x)$ is the boundary of $f^{n+m} W_\delta^u(x)$.
We can choose a maximal set 
$S= \{x_1, \cdots, x_{N(n+m,\epsilon)}\}$ with the additional property that 
$d_u(x_i,x_j) > \epsilon/2$ for $i \neq j$ and $x_i \in A_\epsilon$.
By our choice of $S$ we have that 
$$
A_{\epsilon} \subset \bigcup_{i=1}^{N(n+m,\epsilon)} B_{d_u}(x_i,  \epsilon/2).
$$
By the triangle inequality we have that 
$$
f^{n+m} W_\delta^u(x) \subset \bigcup_{i=1}^{N(n+m,\epsilon)} B_{d_u}(x_i,  \epsilon).
$$
Since by the additional property
$B_{d_u}(x_i, \epsilon/4)\cap B_{d_u}(x_j, \epsilon/4) = \emptyset$ for $i \neq j$
we have that the disjoint union satisfies 
$$
\bigcup_{i=1}^{N(n+m,\epsilon)} B_{d_u}(x_i,  \epsilon/4) \subset f^{n+m} W_\delta^u(x).
$$




We assume without loss of generality that 
$$f^n: f^{m}W^u_\delta(x)
\to f^{n+m}W^u_\delta(x)$$
 locally expands distance along the unstable 
 manifold (which is achieved by our choice of the Riemannian metric being adapted in Definition \ref{parthyp}).
 In particular, we will use the local expansion to show that the primages $y_i := f^{-n} x_i \in  f^m(W_\delta^u(x))$ ($i=1, \cdots, N$)  form an $(n, 2 \epsilon)$-spanning set.

By Lemma \ref{cover0},  for  any point  $z\in \textcolor{black}{X}$  we can choose a point
$y \in f^m(W^u_\delta(x))$ with $z \in W^{cs}_{loc}(y)$ 
and $d(y,z) < \epsilon_0$. Then by Lyapunov stability  for $\epsilon_0>0$ sufficiently small we have $d(f^jz, f^jy) < \epsilon$ for $0\leq j \leq n$.
We can then choose a $y_i$ such that $\textcolor{black}{d_n(y,y_i)} < \epsilon$ since $f^n$ is locally expanding along unstable manifolds.
In particular,  by the triangle inequality
$$d(f^j z, f^jy_i)  \leq d(f^j z, f^jy)+d(f^j y, f^jy_i) \leq 2\epsilon$$
 for $0\leq j \leq n$. Therefore, $\{y_1, \dots, y_{N(n+m,\epsilon)}\}$ is an $(n,2\epsilon)$-spanning set.

Since $G$ is continuous  we have the following bound (cf. Lemma 4.9 in \cite{PaPo}).

\begin{lemma}\label{bound}
For all  $\tau> 0$ there exists $\epsilon > 0$  sufficiently small such that 
 for all $n \geq 1$ and 
points $y_i, z\in X$ satisfying  $d(f^jy_i, f^jz) \leq \epsilon$ for $0 \leq j \leq n-1$
we have $|G^n(y_i) - G^n(z)| \leq n \tau$.
\end{lemma}

It remains to relate $Z_0(n, 2\epsilon)$ to an integral over 
$f^{n+m}W_\delta^u(x)$. By the properties of our choice of 
$\epsilon$-cover for $f^{n+m}W_\delta^u(x)$
we  have that for all $n \geq 1$
$$
\begin{aligned} 
Z_0(n,2\epsilon) &\leq \sum_{i=1}^N \exp(G^n(y_i)) 
\cr
&\leq
\sum_{i=1}^N 
\frac{1}{\lambda(B_{d_u}(x_i,\epsilon/4))}
\int_{B_{d_u}(x_i,\epsilon/4)} \exp(G^n(f^{-n}x_i)) d\lambda(\textcolor{black}{z})\cr
&\leq 
\frac{1}{M} e^{n \tau}  \int_{f^{n+m}W_\delta^u(x)}  e^{G^n(f^{-n}\textcolor{black}{z})}d\lambda(\textcolor{black}{z}) 
\cr
\end{aligned} \eqno{(3.2)}
$$
where $M = M(\epsilon) = \inf_{z} \lambda(B_{d_u}(z,\epsilon/4))>0$.
Finally,    we can bound 
$$ 
\int_{f^{n+m}W_\delta^u(x)}  e^{G^n(f^{-n}\textcolor{black}{z})}d\lambda(\textcolor{black}{z})
\leq 
e^{m \|G\|_\infty}
\int_{f^{n+m}W_\delta^u(x)}  e^{G^{n+m}(f^{-(n+m)}\textcolor{black}{z})}d\lambda(\textcolor{black}{z}). \eqno{(3.3)}
$$
Comparing equations (3.2) and (3.3), we see that 
$$
P(G) = 
\lim_{\epsilon \to  0}
\limsup_{n \to +\infty}
\frac{1}{n} \log 
Z_0(n, 2\epsilon) \leq 
\limsup_{n \to +\infty}\frac{1}{n} \log \int_{f^{n}W_\delta^u(x)}  e^{G^n(f^{-n}\textcolor{black}{z})}
d\lambda(\textcolor{black}{z}) +\tau.
$$
Since $\tau>0$ can be chosen arbitrarily small the lower bound follows.

To prove the reverse inequality in (3.2)
given $\epsilon > 0$ and 
  $n \geq 1$ then we want to
create an $(n,\kappa \epsilon)$-separated  set for some constant $\kappa > 0$.  
To this end, we  can  choose a maximal number of points $x_i \in \textcolor{black}{f^n}W_\delta^u(x)$ $(i = 1, \cdots, N= N(n, \epsilon))$
so that $d_u(x_i, x_j) > \epsilon$ whenever $i \neq j$. 
 We can again assume without loss of generality that 
 $f^n: W^u_\delta(x) \to f^nW^u_\delta(x)$ is locally  distance expanding and thus, in particular, the points
 $y_i = f^{-n}x_i$ 
 ($i=1, \cdots, N = N(n, \epsilon)$)
 form an $(n, \kappa \epsilon)$-separated set, for some $\kappa > 0$ independent of $n$ and $\epsilon$. 
 
The balls $B_{d_u}(x_i,  \epsilon)$ ($i=1, \cdots, N=N(n, \epsilon)$) form a cover for $f^nW^u_\delta(x)$, since otherwise we could choose an extra point  $z \in f^nW^u_\delta(x)$  with   $\inf_i\{d(z, x_i)\} \geq \epsilon$ contradicting the maximality of the $x_i$'s.
We  can therefore use Lemma 
 \ref{bound} to
 bound
 $$
\begin{aligned}
Z_1(n,\textcolor{black}{\kappa} \epsilon) &\geq \sum_{i=1}^N
\frac{e^{-n\tau}}{\lambda(B_{d_u}(x_i, \epsilon))}
 \int_{B_{d_u}(x_i,\epsilon)}  \exp({G^n(f^{-n}z))}d\lambda(z)\cr
&\geq 
\frac{e^{-n\tau }}{L}\int_{f^{n}W_\delta^u(x)}  e^{G^n(f^{-n}\textcolor{black}{z})}d\lambda(\textcolor{black}{z}).  
\end{aligned}
$$
 where  $L = L(\epsilon)= \sup_z  \lambda(B_{d_u}(z, \epsilon)) > 0$.
In particular,  we see that 
$$
P(G) = 
\lim_{\epsilon \to 0}
\limsup_{n \to +\infty}
\frac{1}{n} \log 
Z_1(n, \kappa \epsilon) \geq 
\limsup_{n \to +\infty}\frac{1}{n} \log \int_{f^{n}W_\delta^u(x)}  e^{G^n(f^{-n}\textcolor{black}{z})}d\lambda(\textcolor{black}{z}) - \tau.
$$
Since $\tau > 0$ is arbitrary 
this inequality completes the proof of Proposition \ref{pressure}.
 \end{proof}


\begin{example} \label{example}
If we drop the Lyapunov stability assumption and have no condition on the growth rate of the centre-stable direction then Proposition \ref{pressure} may no longer hold. Let $A \in SL(4,\mathbb{Z})$ and $T_A : \mathbb{T}^4 \rightarrow \mathbb{T}^4$ be the associated diffeomorphism. We consider the product of two Anosov diffeomorphisms on $\mathbb{T}^2$ using the matrix
\begin{equation*}
A = \begin{pmatrix}
3 & 2 &  0 & 0\\
1 & 1 & 0 & 0\\
0 & 0 & 2 & 1 \\
0& 0& 1& 1
\end{pmatrix} 
\end{equation*}
We have eigenvalues $\lambda_1 = 2+ \sqrt{3}$, $\lambda_2 = \frac{3 + \sqrt{5}}{2}$, $\lambda_3 = \frac{3 - \sqrt{5}}{2}$ and $\lambda_4 = 2 - \sqrt{3}$.

We can 
 take $E^u$ to be the direction of the eigenvector corresponding to $\lambda_1$ and $E^{cs}$ to be in the eigenspace generated by the eigenvectors corresponding to $\lambda_{2},$ $\lambda_{3},$ $\lambda_{4}$ to obtain a partially hyperbolic diffeomorphism. Taking $G = 0$ then if Proposition \ref{pressure} applied, it would imply the entropy should be $\log \lambda_1$ when it is clear from the definition that the entropy is the sum of the logarithm of the eigenvalues greater than 1, namely, $\log \lambda_1 + \log \lambda_2$.
\end{example}

\section{Proof of Theorem \ref{thm:gibbsdiffeo_intro}}

In this section we discuss the main result, constructing equilbrium states using pullbacks of measures supported on small pieces of unstable manifold. 




\subsection{Proof of Theorem \ref{thm:gibbsdiffeo_intro}}

The proof of Theorem \ref{thm:gibbsdiffeo_intro} relies on the growth rate result, Proposition \ref{pressure}.


\begin{proof}[Proof of Theorem 1.1]
By Alaoglu's theorem on the weak star compactness of the 
space of $f$-invariant
probability measures
we can find an $f$-invariant  probability measure, which we denote by   $\mu$,  and 
a  subsequence $n_k$ such that 
the measures, $\mu_n$ in equation (1.3), have a weak star convergent subsequence with 
$\lim_{k \rightarrow \infty} \mu_{n_k} = \mu$.
Moreover, for  any continuous $F: X \to \mathbb R$ we can compare 
$$
\begin{aligned}
\left|\int F d\mu_n - \int F\circ  f d\mu_n\right| &= 
\left|\frac{1}{n}\sum_{k=0}^{n-1} \int F\circ f^k d\lambda_n
- \frac{1}{n}\sum_{k=0}^{n-1} \int F\circ f^{k+1} d\lambda_{n}\right|\cr
&\leq \frac{2\|F\|_\infty}{n} \to 0 \hbox{ as } n \to +\infty\cr
\end{aligned}
$$
and, in particular,  one easily sees that $\mu$ is $f$-invariant.

For convenience we denote
$$
\begin{aligned}
Z_n^G 
&= \int_{W_\delta^u(x)} \exp \left( 
\sum_{k=0}^{n-1} (G-\Phi)(f^ky)
\right)d\lambda_{W_\delta^u(x)}(y)\cr
&
=  \int_{f^nW_\delta^u(x)} \exp \left( 
\sum_{k=1}^{n} G(f^{-k}y)
\right)d\lambda_{f^nW_\delta^u(x)}(y).
\cr
\end{aligned}
$$
We want to show that
$\mu$ is an equilibrium state for $G$.

\begin{definition}
Given a finite partition $\mathcal P = \{P_i\}_{i=1}^N$ we say that it has size $\epsilon > 0$ if $\sup_{i}\left\{ \hbox{\rm diam}(P_i)  \right\}< \epsilon$.
\end{definition}

\textcolor{black}{By Lemma \ref{bound}}, for  any $\tau>0$ and $n \geq 1$ we can choose a partition $\mathcal P$  of size 
$\epsilon$  such that  for all  $x,y \in A \in \vee_{i=0}^{n-1}
T^{-i}\mathcal P$ we have that
$$\left|\sum_{k=0}^{n-1} G(f^{k}x) - \sum_{k=0}^{n-1} G(f^{k}y)\right| \leq n \tau. \eqno(4.1)$$

Proceeding with the proof of  Theorem \ref{thm:gibbsdiffeo_intro}, 
for each $A \in  \bigvee_{h=0}^{n-1} f^{-h} \mathcal P$ we can fix a choice of \textcolor{black}{an} $x_A \in A$.
By definition of $\lambda_n$, for each $0 \leq j \leq n-1$:
\begin{equation*}
\int_{f^jW^u_\delta(x)} G d (\textcolor{black}{f^j_*}\lambda_{n}) = \frac{1}{Z_n^G}  \int_{f^n(W^u_\delta(x))} 
\exp\left({\sum_{k=1}^{n} G(f^{-k}y)}\right)G( f^{-(n-j)}y)  d\lambda_{f^nW^u_\delta}(y).
\end{equation*}
Hence,
using  the definition of $\mu_n$  in (1.\textcolor{black}{3}) and lower bound in equation (4.1) we have that,
$$
\int G(y) d\mu_{n}(y)
						 \geq \frac{1}{nZ_n^G} \sum_{A \in \bigvee_{h=0}^{n-1} f^{-h}\mathcal P} 
			\left( \sum_{j=0}^{n-1}G( f^{j}x_A) - n\tau \right)    \int_{f^n(A \cap W^u_\delta(x))} \exp\left({\sum_{k=1}^n G(f^{-k}y)}\right)  d\lambda_{f^nW^u_\delta}(y).	\eqno(4.2)
$$
We can write for a Borel set $A \subset X$:
$$
\lambda_n(A) = \int_{f^nW^u_\delta(x)} 
\exp\left({\sum_{k=1}^{n}G(f^{-k}y)}\right)
\chi_{f^nA}(y) d\lambda_{f^nW^u_\delta(x)} (y).
$$
Using Lemma \ref{bound} again we have
$$
\begin{aligned}
\log \lambda_n(A) 				&\leq \sum_{k=0}^{n-1}G(f^{k}x_A) + n\tau  + \log \lambda_{f^n W^u_\delta(x))}(f^n(A))
\cr
&\leq \sum_{k=0}^{n-1}G(f^{k}x_A) + n \tau 				
\end{aligned}
\eqno(4.3)
$$
where  in  the last inequality we use that the diameters of elements in the partition are  arbitrarily small so that
  $\log \lambda_{f^n W^u_\delta(x))}(f^n(A))$ is negative (cf. Theorem 1.2 in \cite{PaPo}).

Letting $K_{n,A} =   \int_{f^n(A \cap W^u_\delta(x))} 
\exp\left({\sum_{k=1}^{n}G(f^{-k}y)}\right)
 d\lambda_{f^n W^u_\delta(x)} $
 we can 
 consider the entropy
$$
\begin{aligned}
H_{\lambda_{n}} \bigg(\bigvee_{h=0}^{n-1} f^{-h}\mathcal{P} \bigg) &
 = - \sum_{A \in \bigvee_{h=0}^{n-1} f^{-h}\mathcal{P}} \lambda_{n}(A) \log \lambda_{n}(A) \\
			& = - \sum_{A \in \bigvee_{h=0}^{n-1} f^{-h}\mathcal P} \frac{K_{n,A}}{Z_n^G} \log \frac{K_{n,A}}{Z_n^G} \\
			& = \log Z_n^G - \sum_{A \in \bigvee_{h=0}^{n-1} f^{-h}\mathcal P} \frac{K_{n,A}}{Z_n^G} \log{K_{n,A}}, 
\end{aligned}
\eqno(4.4)
$$
where the last equality uses
$\sum_{A \in \bigvee_{h=0}^{n-1}f^{-h}\mathcal P} K_{n,A} =  Z_n^G.$
Therefore,  comparing  (4.3) and  (4.4) gives
$$
H_{\lambda_{n}} \bigg(\bigvee_{h=0}^{n-1} f^{-h}\mathcal P \bigg) \geq \log Z_n^G  - \sum_{A \in \bigvee_{h=0}^{n-1} f^{-h}\mathcal P} \frac{K_{n,A}}{Z_n^G}
\left(\sum_{k=0}^{n-1}G(f^{k}x_A)+  n \tau\right).
\eqno(4.5)
$$
By  (4.2) we can also bound
$$
n \int_X G d \mu_{n} \geq 
  \frac{1}{Z_n^G} \sum_{A \in \bigvee_{h=1}^{n} f^{-h}\mathcal P} \left(\sum_{k=1}^nG(f^{k}x_A) -  n \tau\right) K_{n,A}.
  \eqno(4.6)
$$
Comparing (4.5) and (4.6) we can write
$$
\begin{aligned}
&H_{\lambda_{n}} \bigg(\bigvee_{h=0}^{n-1} f^{-h}\mathcal P \bigg) + n \int_X G d \mu_{n} 
\cr& \geq \log Z_n^G - \sum_{A \in \bigvee_{h=0}^{n-1} f^{-h}\mathcal{P}} \frac{K_{n,A}}{Z_n^G}\left(\sum_{k=0}^{n-1}G(f^{k}x_A) 
+ n \tau\right)\cr
&\qquad\qquad\qquad
+ \frac{1}{Z_n^G} \sum_{A \in \bigvee_{h=0}^{n-1} f^{-h}\mathcal P} \left(\sum_{k=0}^{n-1}G(f^{k}x_A)-  n \tau\right) K_{n,A} \\
			& \geq \log Z_n^G  - 2 n \tau.
\end{aligned}
\eqno(4.7)
$$
We can use (4.7) and Lemma \ref{Misent} to write, for $0 < q < n$,
$$
\begin{aligned}
q \log Z_n^G - q n \int_X G d \mu_{n}-2q n \tau & \textcolor{black}{\leq}
q H_{\lambda_{n}} \bigg(\bigvee_{h=0}^{n-1} f^{-h}\mathcal{P} \bigg) , \\ 
				&	 \leq n H_{\mu_{n}} \bigg(\bigvee_{i=0}^{q-1} f^{-i}\mathcal{P} \bigg) + 2q^2|\mathcal{P}|, 
\end{aligned}
$$
which we can rearrange to get
$$
\begin{aligned}				
\frac{\log Z_n^G}{n} - {2\tau} - \frac{2q|\mathcal{P}|}{n} & \leq \frac{H_{\mu_{n}} \bigg(\bigvee_{i=0}^{q-1} f^{-i}\mathcal{P} \bigg)}{q} + \int_X G d \mu_{n}.
\end{aligned}
$$
Letting $n_k \to +\infty$ gives that 
\begin{align*}
P(G) & = \lim_{k \rightarrow \infty} \frac{\log Z_{n_k}^G}{n_k}\cr
& \leq \lim_{k \rightarrow \infty} \left(\frac{H_{\mu_{n_k}}\bigg(\bigvee_{i=0}^{q-1} f^{-i}\mathcal{P} \bigg)}{q} + \int_X G d \mu_{n_k}\right) + 2 \tau \\
			& =  \frac{H_{\mu}\bigg(\bigvee_{i=0}^{q-1} f^{-i}\mathcal{P} \bigg)}{q} + \int_X G d \mu + 2\tau,
\end{align*}
where we assume without loss of generality that the boundaries of the partition have zero measure.
Letting  $q \rightarrow \infty$,
$$
\begin{aligned}
P(G) \leq \textcolor{black}{h_{\mu}(\mathcal P)} + \int_X G d \mu + \textcolor{black}{2}\tau.
\end{aligned}\eqno(4.8)
$$
Finally, we recall that  $\tau$ is arbitrary.
Therefore, since $\mu$ is an $f$-invariant probability measure 
we see from the variational principle (1.1) that the inequalities in (4.8) are actually equalities 
(since \textcolor{black}{$h_{\mu}(\mathcal P)  \leq  h({\mu})$})
and therefore  we conclude  that the
measure  $\mu$ is an equilibrium state for $G$.

\end{proof}

\begin{remark}
If an equilibrium state corresponding to a potential $G$ were known to be unique then we would  have that $\mu_n \to \mu_G$
as  $n \to +\infty$.
\end{remark}

\section{Examples}
We can consider some specific examples.

\begin{example}[Measure of Maximal Entropy] \label{mme}
In  the special case where $G = 0$, the weak* limit points are measures of maximal entropy, $\mu_{MME}$. 
In particular, (1.1) now reduces to  $P(0) = h_{top}(f)$, the topological entropy.
Furthermore,  the sequence of densities  $(\lambda_n)_{n=1}^\infty$
 is given by
$$
\begin{aligned}
\frac{d\lambda_n}{d\lambda}(y) &:=  
\frac{
 \exp\left( -\sum_{i=0}^{n-1} \textcolor{black}{\Phi}(f^iy) \right) 
 }{
 \int_{W_\delta^u(x)}
 \exp\left( -\sum_{i=0}^{n-1} \textcolor{black}{\Phi}(f^iz) \right) 
 d\lambda_{W_\delta^u(x)}(z)
 }\cr
&=
\frac{
\left|\det ( Df^n|E_y^u)\right|
 }{
 \int_{W_\delta^u(x)}
\left|\det (Df^n|E_z^u)\right|
 d\lambda_{W_\delta^u(x)}(z)
 }
 \hbox{ for  $y \in  W^u_\delta(x)$.}
\cr
\end{aligned}
$$
The averages become
$$
\mu_n = 
\frac{1}{n}\sum_{k=0}^{n-1} f_*^k
\left(\frac{
\left|\det ( Df^n|E_y^u)\right|
 }{
 \int_{W_\delta^u(x)}
\left|\det (Df^n|E_z^u)\right|
  d\lambda_{W_\delta^u(x)}(z)
 }  \lambda \right), \quad n \geq 1,
$$ 
and 
the weak* limit points are measures of maximal entropy.
\end{example}


\begin{example}[u-Gibbs measure] \label{u-Gibbs}
If we consider the geometric potential $G = \Phi = -  \log \left|\det ( Df|E_x^u)\right|$ then Theorem \ref{thm:gibbsdiffeo_intro} also shows that $u$-Gibbs measures are equilibrium states for the geometric potential. The weights in Theorem \ref{thm:gibbsdiffeo_intro} reduce to
$$
\frac{d\lambda_n}{d\lambda}(y) =  \exp\left( \sum_{i=0}^{n-1} (G - \Phi)(f^iy)\right)
 =1
\hbox{ for all $y \in W^u_\delta(x)$}
 $$
so
$\lambda_n = \lambda$, for all $n \geq 1$.  In particular, 
$$
\mu_n = 
\frac{1}{n}\sum_{k=0}^{n-1} f_*^k \lambda_n = \frac{1}{n}\sum_{k=0}^{n-1} f_*^k \lambda, 
$$
for $n \geq 1$. The weak* limit points of $\mu_n$ are u-Gibbs measures.
\end{example}

\begin{example}
We can apply Theorem \ref{thm:gibbsdiffeo_intro} to the geometric q-potential,
\begin{equation*}
\Phi_q = - q \log \left|\det ( Df|E_x^u)\right| \text{ \hspace{5mm} for $q \in \mathbb{R}$}.
\end{equation*}
The sequence of densities are then 
$$
\begin{aligned}
\frac{d\lambda_n}{d\lambda}(y) &:= 
\frac{
\left|\det ( Df^n|E_y^u)\right|^{1-q}
 }{
 \int_{W_\delta^u(x)}
\left|\det (Df^n|E_z^u)\right|^{1-q}
 d\lambda_{W_\delta^u(x)}(z)
 }
 \hbox{ for  $y \in  W^u_\delta(x)$.}
\cr
\end{aligned}
$$
The averages become
$$
\mu_n = 
\frac{1}{n}\sum_{k=0}^{n-1} f_*^k
\left(\frac{
\left|\det ( Df^n|E_y^u)\right|^{1-q}
 }{
 \int_{W_\delta^u(x)}
\left|\det (Df^n|E_z^u)\right|^{1-q}
  d\lambda_{W_\delta^u(x)}(z)
 }  \lambda \right), \quad n \geq 1.
$$ 
Notice that if $q = 0$ we recover Example \ref{mme} and $q = 1$ reduces to Example \ref{u-Gibbs}.
\end{example}

\section{Partially Hyperbolic Diffeomorphisms with Subexponential Contraction in Centre-Unstable Manifolds} \label{PHDwSCiCUM}

In this section we explore constructions of equilibrium states for partially hyperbolic diffeomorphisms with a stable, centre-unstable splitting. Our motivation is to weaken the Lyapunov stability condition in Theorem \ref{thm:gibbsdiffeo_intro} in this context. Again we construct equilibrium states using a suitable sequence of reference measures but this time supported on a local centre-unstable manifold. We still require an additional condition, this time restricting the contraction of orbits in centre-unstable manifolds

\begin{definition}[Partially hyperbolic set with stable, centre-unstable splitting] \label{parthyp_cu}
Let $M$ be a compact Riemannian manifold and $f : M \rightarrow M$ be an attracting $C^1$ diffeomorphism. A closed $f$-invariant subset $X \subset M$ is said to be partially hyperbolic if:
\begin{itemize}
\item there exists a continuous splitting of the tangent bundle into subbundles $E^{cu}$ and $E^{s}$ such that $T_{X}M = E^{cu} \oplus E^{s}$; and
\item there is a Riemannian metric $|| \cdot ||$ on $M$, and constants $0 < \lambda_1 < \lambda_2$ with $0 < \lambda_1 < 1$ such that, for every $x \in X$,
\begin{align*}
|| Df_xv|| & \leq \lambda_1 ||v|| \hbox{ \hspace{5mm} for $v \in E_x^s$}, \\
|| Df_xv|| & \geq \lambda_2 ||v|| \hbox{ \hspace{5mm} for $v \in E_x^{cu}$}.
\end{align*}
\end{itemize}
\end{definition}


We have the following analogue of Lemma \ref{stab} in the case of partially hyperbolic diffeomorphisms with a stable, centre-unstable splitting.


\begin{thm}[Shub \cite{shub}] \label{smtpart}
    For any $x \in X$ and for $\delta > 0$ sufficiently small, there are two $C^1$ embedded discs, $W_\delta^s$ and $W_\delta^{cu}$ tangent to $E_x^s$ and $E_x^{cu}$ respectively and satisfying,
    \begin{enumerate}
        \item Choose $\rho > 0$ such that $\lambda_1 < \rho < \lambda_2$. Then the $\delta$-local stable manifold is given by, 
        \begin{equation*}
            W_\delta^s(x) = \{y \in M \hbox{ : } d(f^{n}x, f^{n}y) \leq \delta, \forall n \geq 0 \hbox{ and } d(f^{n}x, f^{n}y) / \rho^n \rightarrow 0 \hbox{ as }n \rightarrow \infty \};
        \end{equation*}
        \item $f(W_\delta^s(x)) \subset W_\delta^s(f(x))$ and $f$ contracts distances by a constant close to $\lambda_1$;
        \item $f(W_\delta^{cu}(x)) \cap B(x,\delta) \subset W_\delta^{cu}(f(x))$.
    \end{enumerate}
\end{thm}

Our results require an additional condition which is less restrictive than Lyapunov stability. Heuristically, we assume that the contraction of distances in the centre-unstable manifold needs to be subexponential.

\begin{definition}[Subexponential contraction in the centre-unstable manifold] \label{secicum}
    A partially hyperbolic diffeomorphism with stable, centre-unstable splitting $f : X \rightarrow X$ satisfies \emph{subexponential contraction} in the centre-unstable manifold if there exists an increasing $g : \mathbb{N} \rightarrow \mathbb{R}_{\geq 1}$ such that for $x \in X$, $y, z \in W_\delta^{cu}(x)$ and $n \in \mathbb{N}$
    \begin{equation*}
        d(f^{n}y, f^{n}z) > \frac{1}{g(n)}d(y,z),
    \end{equation*}
    with 
    \begin{equation*}
        \limsup_{n \rightarrow \infty} \frac{1}{n} \log g(n) = 0.
    \end{equation*}
\end{definition}


In Remark \ref{counter} we will discuss how  Example \ref{example} is again a counterexample to our results if we do not assume the property in Definition \ref{secicum}.

\subsection{Growth of centre-unstable manifolds} \label{GoC-UM}

As in Proposition \ref{pressure} we need the following characterisation of the pressure in terms of growth rates of appropriately weighted centre-unstable manifolds. 

\begin{proposition}\label{pressure_cu}
Let $f: X \to X$ be a mixing partially hyperbolic attracting diffeomorphism with subexponential contraction in centre-unstable manifolds. For any continuous function $G: X \to \mathbb R$, $x \in X$ and $\delta > 0$ sufficiently small, we have 
\begin{equation} \label{growtheqpart}
P(G) = \limsup_{n \to +\infty} \frac{1}{n}\log  \int_{f^nW_\delta^{cu}(x)} e^{S_n G(f^{-n}y)}  d\lambda_{f^nW_\delta^{cu}(x)}(y). 
\end{equation}
\end{proposition}

In Proposition \ref{pressure} we require Lyapunov stability to restrict growth in the centre-stable direction so that the centre direction does not contribute to the pressure. 
In the present  case we require that the centre-unstable direction does not contract exponentially, otherwise this 
contributes to the right hand side of (\ref{growtheqpart}).

Consider the centre-unstable geometric potential which is defined as the exponential growth of the derivative of $f$ restricted to the centre-unstable bundle. 


The proof of Proposition \ref{pressure_cu} is similar to the proof of Proposition \ref{pressure}. The overall approach is the same but we have to be more careful constructing the spanning set in the first half of the proof. In particular we no longer cover $f^{n+m}W_\delta^{cu}(x)$ with balls of a fixed radius.

\begin{proof}
We start with the following analogue of Lemma \ref{cover0}.

\begin{lemma}\label{cover_par}
For any $\epsilon > 0$ there exists an $m > 0$ such that $f^mW^{cu}_\delta(x)$ is $\epsilon$-dense in $X$.  In particular, we can assume that
 $X = \cup_{y \in f^mW^{cu}_\delta(x)} W_\epsilon^s(y).$
\end{lemma}

This is a direct consequence of mixing and the local product structure.

To get a lower bound on the growth rate in Proposition \ref{pressure_cu}, given $\epsilon > 0$ and $n \geq 1$  we want to construct an $(n, 2\epsilon)$-spanning set. The way we construct the spanning set is similar to Proposition \ref{pressure} although here the cover of $f^{n+m}W_\delta^{cu}(x)$ uses balls whose radius depends on $n$.

We begin by choosing a covering of $ f^{n+m} W_\delta^{cu}(x)$ by balls of radius $\epsilon_n = g(n)^{-1}\epsilon < \epsilon$, 
$$
B_{d_{cu}}(x_i,  \epsilon_n) \hbox{ : } i=1, \cdots, N:=N(n+m, \epsilon_n)$$
contained within the centre-unstable manifold with respect to  the induced metric $d_{cu}$ and let  $A_{\epsilon_n} : = f^{n+m} W_\delta^{cu}(x) \setminus \bigcup_{y \in \partial f^{n+m} W_\delta^{cu}(x)} B_{d_{cu}}(y, \epsilon_n/2)$, where $\partial f^{n+m} W_\delta^{cu}(x)$ is the boundary of $f^{n+m} W_\delta^{cu}(x)$. We can choose a maximal set $S= \{x_1, \cdots, x_{N(n+m,\epsilon_n)}\}$ with the additional properties that $d_{cu}(x_i,x_j) > \epsilon_n/2$ for $i \neq j$ and $x_i \in A_{\epsilon_n}$.
By our choice of $S$ we have that 
$$
A_{\epsilon_n} \subset \bigcup_{i=1}^{N(n+m,\epsilon_n)} B_{d_{cu}}(x_i,  \epsilon_n/2).
$$
By the triangle inequality we have that 
$$
f^{n+m} W_\delta^{cu}(x) \subset \bigcup_{i=1}^{N(n+m,\epsilon_n)} B_{d_{cu}}(x_i,  \epsilon_n).
$$
Since 
$B_{d_{cu}}(x_i, \epsilon_n/4)\cap B_{d_{cu}}(x_j, \epsilon_n/4) = \emptyset$ for $i \neq j$
we have that the disjoint union satisfies 
$$
\bigcup_{i=1}^{N(n+m,\epsilon_n)} B_{d_{cu}}(x_i,  \epsilon_n/4) \subset f^{n+m} W_\delta^{cu}(x).
$$

By Lemma \ref{cover_par}, for  any point  $z\in X$  we can choose a point $y \in f^m(W^u_\delta(x))$ with $z \in W^{s}_{\epsilon}(y)$ and $d(y,z) < \epsilon$. Then $d(f^jz, f^jy) < \epsilon$ for $0\leq j \leq n$. By construction, we can then choose an $x_i$ such that $f^ny \in W_\delta^{cu}(x_i)$ and $d(f^ny,x_i) < \epsilon_n =\frac{1}{g(n)}\epsilon$. Additionally, for every $0 \leq j \leq n$, we have  $d(f^{n-j}y,f^{-j}x_i) < \epsilon$. Otherwise, let $j \in \{0,\dots,n-1\}$ be the smallest such that $d(f^{n-j}y,f^{-j}x_i) > \epsilon$. Then $d(f^ny,x_i) < \frac{1}{g(n)}\epsilon < \frac{1}{g(n)}d(f^{n-j}y,f^{-j}x_i) < \frac{1}{g(j)}d(f^{n-j}y,f^{-j}x_i)$ (as $g$ is increasing). Since $f^{n-j}y$ and $f^{-j}x_i$ are in the same piece of centre-unstable manifold then this would contradict the assumption of Definition \ref{secicum}. In particular, letting $y_i = f^{-n}x_i$,  by the triangle inequality 
$$d(f^j z, f^jy_i)  \leq d(f^j z, f^jy)+d(f^j y, f^jy_i) \leq 2\epsilon$$
 for $0\leq j \leq n$. Therefore, $\{y_1, \dots, y_{N(n+m,\epsilon_n)}\}$ is an $(n,2\epsilon)$-spanning set.

It remains to relate $Z_{0,G}(n, 2\epsilon)$ to an integral over 
$f^{n+m}W_\delta^{cu}(x)$. By the properties of our choice of 
$\epsilon_n$-cover for $f^{n+m}W_\delta^{cu}(x)$ and using Lemma \ref{bound}, we  have that for all $n \geq 1$,
\begin{align}  \label{upperpart1}
Z_{0,G}(n,2\epsilon)  &\leq \sum_{i=1}^N  \frac{1}{\lambda_{f^{n+m}W_\delta^{cu}(x)}(B_{d_{cu}}(x_i,\epsilon_n/4))}
\int_{B_{d_{cu}}(x_i,\epsilon_n/4)} e^{S_nG(f^{-n}x_i)} d\lambda_{f^{n+m}W_\delta^{cu}(x)}(z) \cr 
		&\leq \frac{1}{M(n)} e^{n \tau}  \int_{f^{n+m}W_\delta^{cu}(x)}  e^{S_nG(f^{-n}\textcolor{black}{z})} d\lambda_{f^{n+m}W_\delta^{cu}(x)}(\textcolor{black}{z}) \cr
\end{align} 
where $M(n) = M(n,\epsilon_n) = \inf_{z} \lambda_{f^{n+m}W_\delta^{cu}(x)}(B_{d_{cu}}(z,\epsilon_n/4))>0$. For $\epsilon$ small, there is a constant $K_1$ such that $\lambda(B_{d_{cu}}(z,\epsilon_n/4)) \geq K_1 g(n)^{-dim(E^{cu})} \lambda(B_{d_{cu}}(z,\epsilon/4))$. Therefore, by the defining property of $g$, $\limsup_{n \rightarrow \infty} \frac{- \log M(n)}{n} = 0$.
Finally, we can bound 
\begin{equation} \label{upperpart2}
\int_{f^{n+m}W_\delta^{cu}(x)}  e^{S_nG(f^{-n}\textcolor{black}{z})}d\lambda(\textcolor{black}{z})
\leq 
e^{m \|G\|_\infty}
\int_{f^{n+m}W_\delta^{cu}(x)}  e^{S_{n+m}G(f^{-(n+m)}\textcolor{black}{z})}d\lambda(\textcolor{black}{z}).
\end{equation}
Comparing equations (\ref{upperpart1}) and (\ref{upperpart2}), we see that 
\begin{equation*}
\lim_{\epsilon \to  0}
\limsup_{n \to +\infty}
\frac{1}{n} \log 
Z_{0,G}(n, 2\epsilon) \leq 
\limsup_{n \to +\infty}\frac{1}{n} \log \int_{f^{n}W_\delta^{cu}(x)}  e^{S_nG(f^{-n}\textcolor{black}{z})}
d\lambda(\textcolor{black}{z}) +\tau.
\end{equation*}
Since $\tau>0$ can be chosen arbitrarily small the lower bound follows.

To get an upper bound on the growth rate in Proposition \ref{pressure_cu}, we proceed the same way as in the second half of the proof of Proposition \ref{pressure} so we shall omit it.
\end{proof}

\begin{remark}\label{counter}
      In Example \ref{example}, let the centre-unstable direction be given by the span of the eigenvectors, $\lambda_1, \lambda_2,\lambda_3$, where $\lambda_2 = \frac{1}{\lambda_3}$. Therefore, taking $G=0$ again, if Proposition \ref{pressure_cu} applied it would imply the topological entropy should be $\log \lambda_1 + \log \lambda_2 + \log \lambda_3 = \log \lambda_1$. 
      \end{remark}


\subsection{Constructing equilibrium states}

We conclude with a statement about the construction of equilibrium states for partially hyperbolic systems with subexponential contraction in the centre-unstable direction.

\begin{thm}\label{thm:gibbsdiffeo_cu}
Let $f: X \to X$ be a 
 topologically mixing partially hyperbolic attracting diffeomorphism satisfying Definition \ref{secicum} and  let  $G: X \to \mathbb R$ be a continuous function. Given $x\in X$ and $\delta > 0$ small, consider 
the  sequence of probability measures $(\lambda_n)_{n=1}^\infty$ supported on $W^u_\delta(x)$ 
and absolutely continuous with respect to the induced volume $\lambda_{W^u_\delta(x)}$ given by,
$$
 \lambda_n(A) = 
 \frac{\int_{f^{\textcolor{black}{n}}(A)} \exp \left( \sum_{i=1}^{n} G(f^{-i}y)\right) d\lambda_{f^nW_\delta^u}(y)}
 {\int_{f^nW_\delta^u(x)} \exp \left( \sum_{i=1}^{n} G(f^{-i}z)\right) d\lambda_{f^nW_\delta^u}(z)}
 \hbox{  for Borel  $A \subset W_\delta^u(x)$}.
$$
Then the weak* limit points of the averages  
$$\mu_n :=\frac{1}{n}\sum_{k=0}^{n-1} f^k_* \lambda_n, \quad n \geq 1,
$$ are equilibrium states for $G$.
\end{thm}

The proof of Theorem \ref{thm:gibbsdiffeo_cu} is the same as that of Theorem \ref{thm:gibbsdiffeo_intro} so we will not repeat it.
The next example is a construction of equilibrium states for the centre-unstable geometric potential, the exponential growth of the derivative in the centre-unstable direction. Notice that the measures constructed here are equilibrium states but they may no longer be $u$-Gibbs measures as defined in \cite{ps}.

\begin{example} \label{not u-Gibbs}
If we consider the centre-unstable geometric potential $G = \Phi = -  \log \left|\det ( Df|E_x^{cu})\right|$ then the weights are constant and so $\lambda_n = \lambda_{W_\delta^{cu}(x)}$ for all $n \geq 1$. In particular, 
$$
\mu_n = 
\frac{1}{n}\sum_{k=0}^{n-1} f^k_* \lambda_n = \frac{1}{n}\sum_{k=0}^{n-1} f^k_* \lambda, 
$$
for $n \geq 1$. The weak* limit points of $\mu_n$ are equilibrium states for the geometric potential.
\end{example}

\end{document}